\documentclass{amsart}
\usepackage{enumerate}
\usepackage{cite}
\usepackage{xcolor}
\usepackage{amssymb,amsxtra,amstext,amsbsy,latexsym,mathrsfs}    
\usepackage{amsfonts,euscript,delarray,enumitem}
\usepackage{fancyhdr,nameref,amscd}
\usepackage[all,cmtip]{xy}
\usepackage[letter,center]{crop}
\usepackage{geometry}
\usepackage{tikz,tikz-cd}
\usetikzlibrary{arrows.meta} 
\tikzcdset{every label/.append style = {font = \small}}
\usetikzlibrary{matrix,arrows}
\usepackage{todonotes}
\usepackage{calligra}
\DeclareFontShape{T1}{calligra}{m}{it}{<-> s*[1.07] callig15}{}

\newcommand{\CC}{{\mathbb{C}}}

\newcommand{\QQ}{{\mathbb{Q}}}
\newcommand{\RR}{{\mathbb{R}}}

\newcommand{\ZZ}{{\mathbb{Z}}}

\newcommand{\smvee}{\raise0.9ex\hbox{$\scriptscriptstyle\vee$}}

\newcommand{\Ss}[1]{\mathcal{O}_{#1}}

\newcommand{\Rs}{\tilde{X}}

\newcommand{\Sf}[1]{\mathcal{#1}}
\newcommand{\dimc}[1]{\dim_{\mathbb{C}}\left(#1\right)}

\DeclareMathOperator{\im}{Im}

\DeclareMathOperator{\Exts}{\mathscr{E}\text{\kern -3pt {\calligra\large xt}}\,}
\DeclareMathOperator{\Homs}{\mathscr{H}\text{\kern -3pt {\calligra\large om}}\,}
\DeclareMathOperator{\Hom}{\text{Hom}\,}
\DeclareMathOperator{\Ext}{\text{Ext}\,}
\DeclareMathOperator{\Ders}{\mathscr{D}\text{\kern -3pt {\calligra\large er}}\,}
\DeclareMathOperator{\Der}{Der\,}
\newtheorem{theorem}{Theorem}[section]

\newtheorem{proposition}[theorem]{Proposition}
\newtheorem{corollary}[theorem]{Corollary}
\newtheorem{definition}[theorem]{Definition}

\newtheorem{remark}[theorem]{Remark}

\newtheorem{question}[theorem]{Question}

 \newtheorem*{theorem*}{Theorem}
 \newtheorem*{corollary*}{Corollary}
\newtheorem{question*}{Question}

\title[The blow-up of a singularity at the module of derivations]{The blow-up of a singularity at the module of derivations}

\author[P.~Barajas]{Paul Barajas}
\address{Instituto de Matem\'aticas, Unidad Cuernavaca, Universidad Nacional Aut\'onoma
de M\'exico, Avenida Universidad s/n, Colonia Lomas de Chamilpa, Cuernavaca,
Morelos, México}
\curraddr{}
\email{paul.barajas@im.unam.mx}
\thanks{}
\author[E.~Chávez-Martínez]{Enrique Chávez-Martínez}
\address{Instituto de Ingeniería y Tecnología, UACJ, Ciudad Juárez, Chihuahua, México}
\curraddr{}
\email{enrique.chavez@uacj.mx}
\thanks{}
\author[A.~Romano Vel\'azquez]{Agust\'in Romano-Vel\'azquez}
\address{Instituto de Matem\'aticas, Unidad Cuernavaca, Universidad Nacional Aut\'onoma
de M\'exico, Avenida Universidad s/n, Colonia Lomas de Chamilpa, Cuernavaca,
Morelos, Mexico}
\curraddr{}
\email{agustin.romano@im.unam.mx}
\thanks{}

\date{}

\begin{document}
\begin{abstract}
    We study the problem of resolving singularities via the blow-up of the module of derivations. Our main results are a positive answer for the case of curves and log-canonical surface singularities, i.e., a finite sequence of blow-ups along the module of derivations resolves the singularities of such varieties. 
\end{abstract}
\maketitle

\section*{Introduction}
The problem of resolving singularities has long been a central topic in algebraic geometry. The foundational result in this area is Hironaka’s theorem, which establishes the existence of a resolution of singularities in characteristic zero \cite{Hiro1}. The construction in Hironaka’s theorem is based on a sequence of choices of center of blowups, which raises the natural question whether a simple global and canonical procedure for constructing the resolution can be found. While positive answers are known in specific situations, no general method exists for an arbitrary variety.

Among the approaches developed over the years, the Nash blow-up—introduced by J. Nash and, independently, by J. G. Semple—stands out as a conceptually natural transformation, replacing the singular points of a variety by limits of tangent spaces. The first natural question is whether iterating the Nash blow-up, possibly alternating with normalization, yields a resolution of singularities. The answer is affirmative for curves \cite{Nob}, for surfaces when alternated with normalization \cite{Mark}, and for other families of varieties over fields of characteristic zero \cite{Atana, D, DJ, Hiro2, GT, GS1, GS2, Reba}. Nevertheless, recent work has produced counterexamples in dimensions four and higher \cite{CDLL}. Hence, the general problem of finding a simple canonical procedure for constructing a resolution of singularities remains open.

In this article, we propose a different approach based on the observation that the Nash blow-up is closely related to the module of Kähler differentials. In fact, the Nash blow-up can be described as the blow-up of the ideal of minors of a suitable submatrix of the Jacobian matrix\cite{Nob}. In this work, we propose to study the blow-up of the dual of the module of Kähler differentials, namely the module of derivations. Our main objective is to study the following question:
\begin{question*}
    Let \(X\) be an algebraic variety over a perfect field \(k\). Suppose that we iteratively blow-up of \(X\) at the module \(\Ders(\Ss{X})\) of derivations. Does this process lead to a resolution of singularities after finitely many steps?
\end{question*}

 This gives rise to a new framework where one investigates whether such blow-ups can provide a resolution of singularities. It is important to remark that our main problem is related to the Zariski-Lipman conjecture: Let $X$ be a variety over a perfect field $k$ of any characteristic. The conjecture claims that if $\mathrm{Der}_{\CC}(\Ss{X})$ is free, then $X$ is necessarily smooth. Notice that if there is a counterexample to the Zariski-Lipman conjecture, then the answer to the previous question is negative. 

Our first result is a Nobile-like theorem (see Theorem \ref{Th:Nobile}):
\begin{theorem*}
    Let $k$ be a perfect field and $X$ be an irreducible algebraic normal variety over $k$. Let
    \begin{equation*}
        f \colon \mathrm{Bl}_{\Der}(X) \to X,
    \end{equation*}
    be the blow-up of $X$ at the module $\Ders(\Ss{X})$. Assume that the Zariski-Lipman conjecture holds. If the blow-up morphism $f$ is an isomorphism, then $X$ is non-singular.
\end{theorem*}

It is known that any log-canonical surface singularity is either rational, simple elliptic, or cusp (see, e.g., \cite[Example 6.3.33(d)]{NeBook}). Our main result is an affirmative answer for the case of complex analytic log-canonical surface singularities:
\begin{theorem*}
Let $(X,x)$ be a complex analytic normal surface singularity belonging to one of the following classes: rational singularities, simple elliptic singularities, or cusp singularities. In the rational case, $X$ admits a resolution by a finite iteration of blow-ups at the module of derivations. Furthermore, the resolution is the minimal resolution of singularities. In the simple elliptic case, the normalized blow-up at the module of Zariski differentials $\left(\Omega^1_X\right)^{\smvee \smvee}$ is the minimal resolution. In the cusp case, the normalized blow-up at the module of Zariski differentials or at the module of derivations has only rational singularities; these remaining singularities are then resolved by the rational case.
\end{theorem*}
The guiding idea behind the proof is the geometric McKay correspondence for
rational surface singularities, in the form developed by Artin-Verdier \cite{AV}, Esnault \cite{Es}, Riemenschneider \cite{Rie}
and Wunram \cite{Wu}: indecomposable special reflexive $\mathcal{O}_X$-modules
correspond bijectively to the irreducible components of the exceptional divisor of
the minimal resolution $\pi \colon \Rs \to X$, the reflexive $\mathcal{O}_X$-module $M$ being matched with the component $E_j$ for
which its first Chern class $c_1\bigl((\pi^*M)^{\vee\vee}\bigr)$ satisfies $c_1\bigl((\pi^*M)^{\vee\vee}\bigr)\cdot E_j = 1$. This correspondence is what
makes our method effective: by Theorem~\ref{Theo:IntroBlowUpMMinAdapt} each normalized blow-up of the
module of derivations contracts precisely the components $E_j$ on which $c_1\bigl((\pi^*M)^{\vee\vee}\bigr)\cdot E_j = 0$, so the correspondence determines, at each step, which
exceptional curves survive. This is the mechanism underlying the case-by-case analysis
of Sections~\ref{Sec:An} and~\ref{Sec:LogC} and the reason the iteration
terminates at the minimal resolution.
The principal obstruction to generalising our result beyond the log-canonical setting is the lack of control of the behaviour of the module of derivations under pullback along blow-ups.

It is important to remark that recent work has shown that iterated Nash blow-ups do not provide a resolution of singularities in dimensions four and higher (see~\cite{CDLL}). The main obstruction comes from the existence of affine charts of the Nash blow-up that are isomorphic to the original variety, preventing any progress under iteration. These counterexamples highlight a fundamental limitation of Nash’s construction, motivating the search for alternative approaches.

The paper is structured as follows. Section \ref{Sec:Pre} reviews the necessary preliminaries on reflexive modules, blow-ups of coherent sheaves, and Kähler differentials. In Section \ref{Sec:Nobile}, we establish a Nobile-like theorem for the blow-up of the module of derivations, and as a direct consequence, we show that a finite sequence of such blow-ups resolves curve singularities. Section \ref{Sec:An} presents an explicit computation of the resolution of $A_n$ singularities via this method, while Section \ref{Sec:LogC} addresses the more general case of log-canonical surface singularities.

\section{Preliminaries}
\label{Sec:Pre}
In this section, we recall basics on full sheaves over normal surface singularities, the blow-up of coherent sheaves, and the Kähler differentials. See~\cite{BrHe,Ne,Ishi} for more details on modules and normal surface singularities.

\subsection{Setting and notation} Throughout this article, we denote by $(X,x)$ the germ of an irreducible complex analytic normal surface singularity, i.e., the germ of a complex surface such that its local ring of functions $\Ss{X,x}$ is integrally closed in its field of fractions. In this situation $X$ has a \emph{dualizing sheaf} $\omega_X$, and let $\omega_{X,x}$ be its stalk at $x\in X$, which is called the \emph{dualizing module} of the ring $\Ss{X,x}$ (see~\cite[Chapter~5~\S~3]{Ishi} for more details).
\begin{definition}
    Let $(X,x)$ be the germ of a complex analytic normal surface singularity. We say that $(X,x)$ has a \emph{Gorenstein} normal singularity if the dualizing module is isomorphic to $\Ss{X,x}$. 
\end{definition}
We denote by $X$ a small representative of the germ $(X,x)$, and by  $U$ the \emph{regular part} of $X$, i.e., $U=X \setminus \{x\}$.

Moreover, in this article, we denote by $(C,0)$ the germ of an irreducible complex analytic curve singularity. 

\subsection{Good resolutions and dual graphs}
Let $(X,x)$ be as before. Let
\begin{equation*}
    \pi \colon \Rs\to X
\end{equation*}
be a \emph{resolution of} $(X,x)$, i.e., a proper holomorphic map from a smooth surface $\Rs$ to a given representative of $(X,x)$ such that $\pi$ restricted to the complement of
 $\pi^{-1}(x)$ is biholomorphic. Sometimes we will require $\pi$ to be a \emph{good resolution}, which means that the exceptional divisor  $E:=\pi^{-1}(x)$ is a normal crossing divisor and each irreducible component of
$E$ is smooth. For any normal surface singularity there is always a good resolution, however, it is
not unique.
The \emph{intersection matrix} $M=(m_{u,v})_{u,v \in V}$ associated to the dual graph $\Gamma$ is the intersection matrix of the curves $\{E_v\}_{v\in V}$, i.e., $m_{u,v}=(E_u,E_v)$.
It is negative definite.
\begin{definition}
The {\em geometric genus} of $X$ is defined to be the dimension as a $\mathbb{C}$-vector space of $H^1(\Rs,\Ss{\Rs})$ for any resolution. We denote the geometric genus by $p_g$.
\end{definition}

\subsection{Cohen-Macaulay modules and reflexive modules}
\label{subsec:MCM}
Let $X$ be a normal variety. Let $\Homs_{\Ss{X}}(\bullet,\bullet)$ be the sheaf theoretic Hom functor. 
The dual of an $\Ss{X}$-module $M$ is denoted by $M^{\smvee}:=\Homs_{\Ss{X}}(M,\Ss{X})$. An $\Ss{X}$-module $M$ is called \emph{reflexive} if the natural homomorphism from $M$ to $M^{\smvee \smvee}$ is an isomorphism. 
\begin{remark}
\label{rmk:dual.reflexive}
    Let $M$ be a finitely generated $\Ss{X}$-module. By \cite[Corollary 1.2]{Har3}, the module $M^{\smvee}$ is a reflexive $\Ss{X}$-module.
\end{remark}

Let $(X,x)$ be the germ of an irreducible complex analytic normal surface singularity. Denote by
$(\Ss{X,x},\mathfrak{m},\kappa)$ the local ring at $x$. Since 
\begin{equation*}
\Ext^1_{\Ss{X,x}}\left(\mathfrak{m}, \omega_{X,x}\right) \cong \Ext^2_{\Ss{X,x}}\left(\kappa, \omega_{X,x}\right) \cong \kappa,    
\end{equation*}
there exists a unique (up to non-canonical isomorphism) non-split exact sequence
\begin{equation}\label{Suc.Fund}
    0 \to \omega_{X,x} \to A \to \mathfrak{m} \to 0
\end{equation}

\begin{definition}
The exact sequence \eqref{Suc.Fund} is called \emph{the fundamental sequence}. The $\Ss{X,x}$-module $A$ is called \emph{the Auslander module or the fundamental module}.
\end{definition}
If $(X,x)$ is a hypersurface singularity, by~\cite[Lemma~1.5]{Yos1} the fundamental module $A$ is isomorphic to the third syzygy of $\kappa$. We will use this description, thus we introduce some notation: We can assume that $\Ss{X,x}\cong R = \mathbb{C}\{x,y,z\}/\langle f \rangle$. Take $f_x, f_y$ and  $f_z$ such that $f=x \cdot f_x + y \cdot f_y +z \cdot f_z$. Then, the minimal free resolution of the residue field $\kappa$ is given as follows:
\begin{equation}
    \dots \to R^4  \xrightarrow{C} R^4 \xrightarrow{D} R^4 \xrightarrow{C} R^4 \xrightarrow{B} R^3 \xrightarrow{A} R \to \kappa \to 0,
\end{equation}
where
\begin{align*}
    &A=(x,y,z)^t, \quad &\quad &B= \begin{pmatrix} 0 & -z & y\\ z & 0 & -x \\ -y & x & 0 \\ f_x & f_y & f_x \end{pmatrix}, \\
&C= \begin{pmatrix} 0 & f_x & -f_y & x \\ -f_x  & 0 & f_x & y \\ f_y  & -f_x  & 0 & z \\  -x & -y  & -z & 0 \end{pmatrix}, \quad &\quad &D= \begin{pmatrix}  0& -z& y & -f_x \\ z & 0 & -x & -f_y \\ -y  & x & 0 & -f_z \\ f_x  & f_y & f_z & 0 \end{pmatrix}.
\end{align*}
\begin{remark}
    By~\cite[Lemma~1.1]{Yos1}, the fundamental module is a reflexive module of rank $2$. In the special case of a hypersurface singularity, the pair of matrices $(C,D)$ is a \emph{matrix factorization} (see \cite{Ei}) of $f$, i.e., $CD=DC=f \cdot Id_4$ where $Id_4$ is the $4 \times 4$ identity matrix.
\end{remark}
\begin{remark}
\label{rmk:hyp}
By \cite[Theorem~3.2]{Martsinkovsky_1}, the hypersurface singularity $(X,x)$ defined by $V(z^n - xy)$ satisfies $A \cong \Omega_X^{\smvee \smvee}$.    
\end{remark}

\subsection{Full sheaves and the minimal adapted resolution} Let $(X,x)$ be the germ of a normal surface singularity and
\begin{equation*}
 \pi \colon \Rs \to X   
\end{equation*}
be a resolution. Recall the following definition of full sheaves as in~\cite[Definition~1.1]{Ka}.
\begin{definition}
An $\Ss{\Rs}$-module $\Sf{M}$ is called \emph{full} if there is a reflexive $\Ss{X}$-module $M$ such that 
$\Sf{M} \cong \left(\pi^* M\right)^{\smvee \smvee}$. We call $\Sf{M}$ the full sheaf associated to $M$.
\end{definition}
\begin{definition}
    Let $\Sf{F}$ be a sheaf on $\Rs$. We say that $\Sf{F}$ is generically generated by global sections if it is generated by global sections except in a finite set of points.
\end{definition}
The following characterization of full sheaves will be very important in the following sections.
\begin{proposition}[{\cite[Proposition~1.2]{Ka}}]\label{fullcondiciones}
A locally free sheaf $\Sf{M}$ on $\Rs$ is full if and only if the following two conditions hold.
\begin{enumerate}
\item The sheaf $\Sf{M}$ is generically generated by global sections.
\item The natural map $H^1_E(\Rs,\Sf{M}) \to H^1(\Rs,\Sf{M})$ is injective.
\end{enumerate}
If $\Sf{M}$ is the full sheaf associated to $M$, then $\pi_* \Sf{M}=M$.
\end{proposition}
\begin{definition}\label{def:especial}
\label{def:espmodule}
Let $M$ be a reflexive $\Ss{X}$-module of rank $r$ and $\Sf{M}$ be the full sheaf associated to $M$. The full sheaf $\Sf{M}$ is called \emph{special} if $\dimc{R^1 \pi_* \left(\Sf{M}^{\smvee}\right)} = rp_g$ where $p_g$ is the geometric genus.
We say that $M$ is \emph{a special module} if for any resolution, the full sheaf associated to $M$ is special.
\end{definition}
\begin{definition}
Let $\Sf{F}$ be a sheaf on $\Rs$. The \emph{base points} of $\Sf{F}$ are the points $p \in \Rs$ such that $s(p)=0$ for all $s\in H^0(\Rs,\Sf{F})$. A component $E_j$ is called a \emph{fixed component} of $\Sf{F}$ if any section $s\in H^0(\Rs,\Sf{F})$ vanishes along $E_j$.
\end{definition}
Let $M$ be a reflexive $\Ss{X}$-module. The minimal adapted resolution associated to $M$ was defined in \cite{BoRo} and later generalized in \cite{Romano1}
\begin{definition}
\label{def:minadap}
Let $M$ be a reflexive $\Ss{X}$-module. The minimal resolution $\pi \colon \Rs\to X$ for which the associated full sheaf $\left(\pi^* M\right)^{\smvee \smvee}$ is generated by global 
sections is called the {\em minimal adapted resolution} associated to $M$. 
\end{definition}
\begin{remark}
\label{remark:Son.iguales}
    Let $L$ be a reflexive $\Ss{X}$-module of rank $1$. The minimal adapted resolution of $L$ is the minimal resolution for which the associated full sheaf does not have base points. 
\end{remark}
\subsection{Derivations and Kähler differentials}
For any $\mathbb{C}$-algebra $A$, we denote by $(\Omega_{A/\mathbb{C}},d)$ \emph{the module of Kähler differentials} (\emph{$\mathbb{C}$-differentials} of $A$), and by \emph{$d$ the universal derivation}. We denote by  $\mathrm{Der}_\mathbb{C}(A,F)$ the $A$-module of  \emph{$\mathbb{C}$-derivations} into the $A$-module $F$.  The module \( \mathrm{Der}_{\CC}(\Ss{X,x}, \Ss{X,x}) \) will be abridged  as  \( \mathrm{Der}_{\CC}(\Ss{X,x}) \).
 By the universal property of the Kähler differentials, we get $\mathrm{Der}_\mathbb{C}(A,F)=\Hom_A(\Omega_{A/\mathbb{C}},F)$.

If we regard $X$ as a normal scheme over $\mathbb{C}$, we can define the sheaf of relative differentials $\Omega_{X/\mathbb{C}}$. The reader can find all the details in \cite[Chapter 2, Section 8]{Har2}.
 As in the affine case, the sheaf of relative differentials satisfies the following universal property: for each $\Ss{X}$-module $\Sf{F}$ there is an isomorphism via the universal differential:
\begin{equation*}
    \Homs_{\Ss{X}}(\Omega_{X/\mathbb{C}},\Sf{F}) \cong \Ders_\mathbb{C}(\Ss{X},\Sf{F}).
\end{equation*}
From now on, we write $\Omega_{X}$ instead of $\Omega_{X/\mathbb{C}}$. Let $p\in \mathbb{Z}$ be a non-negative integer. We denote by
\begin{equation*}
    \Omega^p_{X} = \begin{cases}
\bigwedge^p \Omega_{X} &\text{if $p \geq 1$},\\
\Ss{X} &\text{if $p=0$}.
\end{cases}
\end{equation*}
\begin{remark}
\label{rmk:Der.refl}
    By Remark \ref{rmk:dual.reflexive}, the sheaves $\Ders_\mathbb{C}(\Ss{X})$ and $\Omega_{X/\mathbb{C}}^{\smvee \smvee}$ are reflexive sheaves.
\end{remark}
Let $\pi \colon \Rs\to X$ be a resolution of $(X,x)$. Set
\begin{equation*}
    L:= H_2(\Rs,\ZZ), \quad L':=H_2(\Rs, \partial \Rs, \ZZ) \quad \text{and} \quad L_{\QQ}:=L \otimes \QQ.
\end{equation*}
Let $\ell'_1, \ell'_2  \in L_{\QQ}$ where $\ell'_j= \sum_v l_{jv}'E_v$ for $j\in \{1,2\}$. We consider the partial order in $L_{\QQ}$ given  by $\ell'_1 \geq \ell'_2$ if and only if $l_{1v}' \geq l_{2v}'$ for all $v\in V$. We extend the intersection form $(,)$ from $L$ to $L_{\QQ}$ and denote this extension by $(,)_\QQ$. The sheaf  $\Omega^2_{\Rs}$  is the sheaf of holomorphic $2$-forms. The divisor of zeros of any meromorphic section of $\Omega^2_{\Rs}$ is called \emph{the canonical divisor;}
it is  denoted by $K_{\Rs}$. The rational cycle $Z_K\in L'$ (supported on $E$)
that satisfies $(Z_K , E_v)_{\QQ} = -K_{\Rs} \cdot E_v$ for any $v \in V$ is called \emph{the canonical cycle} (it is independent of the choice of $K_{\Rs}$.)
 In the case of Gorenstein singularities, one proves that $Z_K\in L$; see~\cite{Ne}.
 \begin{definition}
Let $(X,x)$ be a normal surface singularity and consider its minimal good resolution. We say that $X$ is \emph{(numerically) log-canonical} if in this resolution $Z_K\leq E$.
\end{definition}
If $(X,x)$ is 
Gorenstein, there is a Gorenstein $2$-form $\Omega_{\Rs}$ which is meromorphic in $\Rs$, and has neither zeros nor poles in $\Rs\setminus E$; it is
called the {\em Gorenstein form}. Let $div(\Omega_{\Rs})=\sum q_iE_i$ be the divisor associated with the Gorenstein form. The canonical cycle is $Z_K=\sum_i -q_iE_i$.

\subsection{Blowing up at coherent sheaves}
In this article, we use the description of Villamayor~\cite{Villa} of the blow-up at a coherent sheaf. Let $R$ be a domain with quotient field $K$. Fractional ideals are, by definition, finitely generated $R$-submodules of $K$. Two fractional ideals $J_1$ and $J_2$ are isomorphic if and only if there exists some element $k$ in $K$ different from zero such that $J_1=kJ_2$. The norm of a module is defined in the class of all fractional ideals modulo isomorphism as follows:

\begin{definition}[{\cite[p.~123]{Villa}}]
Let $M$ be a finitely generated $R$-module of rank $r$. The \emph{norm} of $M$ is the class
\begin{equation*}
    \| M \|_R := \im\left( \bigwedge^r M \to \bigwedge^r M \otimes K \cong K
    \right)/\sim,
\end{equation*}
where $\sim$ denotes the isomorphism as fractional ideals.
\end{definition}
The blow-up of $R$ at the module $M$ is described as follows:

\begin{theorem}[{\cite[Theorem~3.3]{Villa}}]
\label{Th:BlowUpM}
Let $M$ be a finitely generated $R$-module of rank $r$. There exists a blow-up:
\begin{equation*}
    f\colon \mathrm{Bl}_M(R) \to \text{Spec}(R),
\end{equation*}
with the following properties:
\begin{enumerate}
    \item The sheaf $f^*M / \mathrm{tor}$ is a locally free sheaf of $\Ss{\mathrm{Bl}_M(R)}$-modules of rank $r$.
    \item (Universal property) For any morphism $\sigma \colon Z \to \text{Spec}(R)$ such that $\sigma^* M / \mathrm{tor}$ is a locally free sheaf of $\Ss{Z}$-modules of rank $r$, there exists a unique morphism $\beta\colon Z \to \mathrm{Bl}_M(R)$ such that $f \circ{} \beta = \sigma$.
\end{enumerate}
\end{theorem}
Villamayor constructed the above blow-up under more general assumptions, but in our setting Theorem~\ref{Th:BlowUpM} is sufficient. Furthermore, in our situation $R$ is a domain, hence this blow-up is a proper birational morphism (see \cite{Villa} and \cite{Ros} for details). 
The following results will be used later.
\begin{proposition}[{\cite[Proposition~3.3]{Romano1}}]
\label{Prop:2daconstruccion}
Let $(X,x)$ be the germ of a normal surface singularity. Let $M$ be a reflexive $\Ss{X}$-module and 
\begin{equation*}
    \pi\colon \Rs \to X,
\end{equation*}
be the associated minimal adapted resolution. Let $f \colon Bl_M(X) \to X$ be the blow-up of $X$ at the module $M$ and 
\begin{equation*}
    \rho\colon \widetilde{\mathrm{Bl}_M(X)}_{\text{min}} \to \mathrm{Bl}_M(X),
\end{equation*}
be the minimal resolution of $\mathrm{Bl}_M(X)$. Then, $\widetilde{\mathrm{Bl}_M(X)}_{\text{min}} \cong \Rs$.
\end{proposition}

\begin{theorem}[{\cite[Theorem 4.5]{Romano1}}]
\label{Theo:IntroBlowUpMMinAdapt}
Let $(X,x)$ be the germ of a normal surface singularity. Let $M$ be a reflexive  $\Ss{X}$-module of rank $r$. Let $\pi\colon \Rs \to X$ be the minimal adapted resolution associated to $M$ with exceptional divisor $E$. Let $E_1,\dots,E_n$ be the irreducible components of $E$. Let $\Sf{M}:=(\pi^*M)^{\smvee \smvee}$ be the full sheaf associated to $M$, and denote by $c_1\left(\Sf{M} \right)$ its first Chern class. Then, the normalization of the blow-up of $X$ at $M$ is obtained by contracting the irreducible components $E_i$ such that $c_1\left(\Sf{M} \right) \cdot E_i =0$.
\end{theorem}

\section{Nobile's criterion}\label{Sec:Nobile}

In the following sections, we study the blow-up of the module of derivations in dimensions one and two, under suitable hypotheses. Our guiding example is Nobile’s theorem for the Nash blow-up, which states that the Nash blow-up is an isomorphism only in the nonsingular case (see \cite[Theorem 2]{Nob}). Thus, when the variety is singular, the Nash blow-up necessarily produces a non-trivial modification. The proof of Nobile’s theorem relies on a regularity criterion formulated in terms of the freeness of the module of Kähler differentials (see \cite[Chapter 9]{Matsumura}).

In the same spirit, we investigate whether an analogous statement can be obtained for the module of derivations. Since a regularity criterion based on derivations is closely related to the Zariski–Lipman conjecture, this conjecture provides the natural framework for our approach. Assuming it, we prove the following analogue of Nobile’s theorem for the blow-up associated with the module of derivations:

\begin{theorem}\label{Th:Nobile}
    Let $k$ be a perfect field and $X$ be an integral algebraic variety over $k$. Let
    \begin{equation*}
        f \colon \mathrm{Bl}_{\Der}(X) \to X,
    \end{equation*}
    be the blow-up of $X$ at the module $\Ders(\Ss{X})$. Assume that the Zariski-Lipman conjecture holds. If the blow-up morphism $f$ is an isomorphism, then $X$ is nonsingular.
\end{theorem}
\begin{proof}
    Let $f \colon \mathrm{Bl}_{\Der}(X) \to X$ be the blow-up of $X$ at the module $\Ders(\Ss{X})$. By Theorem~\ref{Th:BlowUpM}, we see that \[     f^*(\Ders_{\mathbb{C}}(\Ss{X}))/\mathrm{tor}(f^*(\Ders_{\mathbb{C}}(\Ss{X}))),     \] is a locally free sheaf of $\Ss{\mathrm{Bl}_{\Der}(X)}$-modules. Since $f$ is an isomorphism, we obtain an isomorphism,
\[     f^*(\Ders_{\mathbb{C}}(\Ss{X}))/\mathrm{tor}(f^*(\Ders_{\mathbb{C}}(\Ss{X}))) \cong \Ders_{\mathbb{C}}(\Ss{X})/\mathrm{tor}(\Ders_{\mathbb{C}}(\Ss{X})),     \] therefore,  $\Ders_{\mathbb{C}}(\Ss{X})/\mathrm{tor}(\Ders_{\mathbb{C}}(\Ss{X}))$ is a locally free $\Ss{X}$-module.

    Since $\Ss{X}$ is a domain and $\Ders_{\mathbb{C}}(\Ss{X})$ is a reflexive $\Ss{X}$-module, we have $\mathrm{tor}(\Ders_{\mathbb{C}}(\Ss{X})) = 0$. Hence, $\Ders_{\mathbb{C}}(\Ss{X})$ is a locally free sheaf of $\Ss{X}$-modules. Let $p \in X$, then
    \[
    \Ders_{\mathbb{C}}(\Ss{X})_p = \Der_{\mathbb{C}}(\Ss{X,p})
    \]
    is a free $\Ss{X,p}$-module. Since we assume the Zariski–Lipman conjecture holds, it follows that $\Ss{X,p}$ is regular. As this holds for every point $p \in X$, we conclude that $X$ is not singular.
\end{proof}

If we already have a criterion that tells us when the blow-up has resolved the singularities, then it is natural to ask the following question:

\begin{question}\label{Q1}
Let \(X\) be an algebraic variety over a perfect field \(k\) of characteristic $p\geq0$. Suppose that we iteratively blow-up of $X$ at the module \(\Ders(\Ss{X})\). Does this process end after finitely many steps?  
\end{question}

\begin{remark}\label{Lipman CE}

The preceding question has a negative answer in positive characteristic due to the failure of the Zariski-Lipman conjecture in this setting. Indeed, Lipman \cite{Lip} provided a counterexample by considering the normal surface $X \subset \mathbb{A}^3_k$ defined by $xy - z^p = 0$ over a perfect field $k$ of characteristic $p>0$. Since the $\mathcal{O}_X$-module $\Der_k(\mathcal{O}_X)$ is free of rank two, generated by $d_1 = x\partial_x - y\partial_y$ and $d_2 = \partial_z$, the blow-up at a free module is trivial and hence $Bl_{\Der}(X) \cong X$.\\
While the classical Nobile theorem also fails in positive characteristic --- for instance, the non-normal variety $X=V(x^3-y^2)\subset\mathbb{A}_k^2$ satisfies $\mathrm{Nash}(X)\cong X$ over $k$ a field of characteristic $2$ or $3$~\cite{Nob} --- Duarte and Núñez-Betancourt~\cite{DNB} showed that normality suffices to recover a Nobile-type theorem for the Nash blow-up in positive characteristic. However, since Lipman's counterexample already involves a normal variety, normality alone is insufficient to guarantee an analogous result for the blow-up of derivations. Consequently, from now on, we will restrict ourselves to varieties defined over the complex numbers and normal if necessary.

\end{remark}

In the case of one-dimensional varieties over fields of characteristic zero, the Zariski-Lipman conjecture holds, and the answer to the question \ref{Q1} is affirmative. This follows from Theorem \ref{Th:Nobile}.

\begin{corollary}\label{Curvas}
    Let $C$ be an algebraic curve over an algebraically closed field of characteristic zero. A finite sequence of blow-ups at the module $\Ders(\Ss{C})$ resolves the singularities of $C$.
\end{corollary}

\begin{proof}
Since $f \colon \mathrm{Bl}_{\Der}(C) \to C$ is a finite birational morphism, by iterating the blow-up process a finite number of times we reach the normalization. Since the normalization is smooth, the process ends by Theorem \ref{Th:Nobile}.
\end{proof}

\section{Blow-up of rational surface singularities}\label{Sec:An}
\subsection{The $A_n$ case}
In this subsection, we will compute in an explicit way the blow-up of the module of the singular surface $A_n$ at the module of derivations. For this, we will use the results of Section~\ref{subsec:MCM}, in particular Remark \ref{rmk:hyp}, and the combinatorial properties of $A_n$ seen as a toric variety. First, we need to describe $A_n$ as a toric variety.

\begin{definition}\label{an-def} Given a semigroup $\Gamma\subset\ZZ^2$, we denote by $X_\Gamma$ the toric variety associated to the monomial algebra $\CC[\Gamma]$, i.e., $X_\Gamma=\text{Spec}(\CC[\Gamma])$. In particular, consider the semigroup
\begin{equation*}
\Gamma_{n}=\left \langle(1,0),(1,1),(n,n+1) \right\rangle_{\ZZ_{\geq0}}\subset \ZZ^{2}.    
\end{equation*}
We denote by $A_{n}$ the normal toric surface corresponding to $\Gamma_n$. The embedding of this variety in $\CC^3$ is given by the equation $A_n=V(xz-y^{n+1})$ (see \cite[Theorem 2.3.1]{NeBook}).
\end{definition}
Now, we need the combinatorial construction of the blow-up of a monomial ideal in a toric variety (for more details of this construction, see \cite[Section 2.6]{GT}).

\textbf{Combinatorial description of monomial blow-ups:} Let $X_{\Gamma}\subset\CC^n$ be a toric variety defined by a semigroup $\Gamma\subset\ZZ^d$. Let $I=\langle x^{m_1},\ldots,x^{m_k}\rangle\subset\CC[X_\Gamma]$ be a monomial ideal.
\begin{enumerate}
\item Abusing the notation, consider $I=\left \{m_1,\ldots,m_k \right\}\subset \Gamma$. Let $\mathcal{N}(I)$ be the Newton polyhedron of $I$, i.e. the convex hull in $\RR^d$ of the set $\bigcup_{1}^{k}\{m_i+\RR_{\geq0}\Gamma\}$.
\item Consider  $m_i$ for some $i\in\{1,\ldots,k\}$. Denote 
\begin{equation*}
 \Gamma_i=\Gamma+ \mathrm{Span}_{\ZZ_{\geq0}} \left  \{m_j-m_i|j\neq i \right\}.   
\end{equation*}
\item Given $m_i,m_j\in I$, the associated toric varieties $X_{\Gamma_{i}}$ and $X_{\Gamma_j}$ can be glued by the principal isomorphic open sets of each variety. The isomorphism between these open sets is induced by the isomorphism of local rings:
$$\CC[\Gamma_{m_i}]_{\frac{x^{m_j}}{x^{m_i}}}\cong\CC[\Gamma_{m_j}]_{\frac{x^{m_i}}{x^{m_j}}}.$$
\item The variety resulting from the above gluing is the blow-up of $X_\Gamma$ along $I$ (see \cite[Proposition 32]{GT}). We denote it by $\mathrm{Bl}_{I}X_\Gamma$.
\item Let $I_0=\{i\in I\,\,|\,\, m_i \text{ is a vertex of } \mathcal{N}(I)\}$. Then,
\begin{equation*}
    \mathrm{Bl}_{I}X_\Gamma=\sqcup_{i\in I_0}X_{\Gamma_i}/\sim,
\end{equation*}
where $\sim$ is the relation given in (3) above.
\end{enumerate}

\begin{theorem}\label{an}
    Let $A_n$ be as in Definition \ref{an-def}. Then 
    $$\mathrm{Bl}_{Der\,}A_n= \begin{cases}
        \CC^2\sqcup A_{n-2}\sqcup\CC^2/\sim       & \text{if } n \geq 3 \\
    \CC^2\sqcup\CC^2\sqcup\CC^2/\sim     & \text{if } n=2 \\
    \CC^2\sqcup\CC^2/\sim    & \text{if } n=1
    \end{cases}
    $$
    In particular, for $n=1$ or $n=2$, $\mathrm{Bl}_{Der\,}A_n$ is smooth.
\end{theorem}

\begin{proof} Since $A_n$ is the zero locus of $f=xz-y^{n+1}$, let $f_x=0$, $f_y=-y^n$ and $f_z=x$. Notice that $f=x \cdot f_x + y \cdot f_y +z \cdot f_z$. Substituting in the matrix $D$ of Section~\ref{subsec:MCM}, we obtain:
$$D= \begin{pmatrix}  0& -z& y & 0 \\ z & 0 & -x & y^n \\ -y  & x & 0 & -x \\ 0  & -y^n & x & 0 \end{pmatrix}.$$
By Remark \ref{rmk:hyp} and \cite[Corollary 5.2]{Martsinkovsky_1}, the blow-up is given by the blow-up of the ideal $I$ that is defined by the $2\times2$ minors of any two columns of $D$ of maximal rank. For this, we will take columns 3 and 4 of the matrix $D$, obtaining
\begin{equation*}
 I=\langle y^{n+1}, x^2, xy\rangle=\langle (uv)^{n+1}, u^2, u^{2}v\rangle\subset \CC[u,uv,u^{n}v^{n+1}]=\CC[X_{\Gamma}].   
\end{equation*}
Since $I$ is a monomial ideal, we can apply the previous algorithm to compute the blow-up at the ideal $I$. In this case $I=\{(2,0),(2,1),(n+1,n+1)\}$. 
\begin{description}
    \item[Description of $\Gamma_1$, where $m_1=(2,0)$] By definition, 
    \begin{align*}
  \Gamma_1&=\Gamma+ \mathrm{Span}_{\ZZ_{\geq0}} \left \{(0,1),(n-1,n+1)\right \}=\mathrm{Span}_{\ZZ_{\geq0}} \left \{(1,0),(1,1),(n,n+1),(0,1),(n-1,n+1)\right \}\\
  &=\mathrm{Span}_{\ZZ_{\geq0}} \left \{(1,0),(0,1)\right \}.      
    \end{align*}
Thus, $X_{\Gamma_1}=\CC^2$. 
\item[Description of $\Gamma_2$, where $m_2=(2,1)$] By definition,
\begin{equation*}
    \Gamma_2=\Gamma+ \mathrm{Span}_{\ZZ_{\geq0}} \left \{(0,-1),(n-1,n)\right \}=\mathrm{Span}_{\ZZ_{\geq0}} \left \{(1,0),(1,1),(n,n+1),(0,-1),(n-1,n)\right \}.
\end{equation*}
 Since $(1,0)=(1,1)+(0,-1)$ and $(n,n+1)=(n-1,n)+(1,1)$, we get 
 \begin{equation*}
  \Gamma_2=\mathrm{Span}_{\ZZ_{\geq0}} \left \{(0,-1),(1,1),(n-1,n)\right \}.
 \end{equation*}
 Thus, $X_{\Gamma_2}=A_{n-2}$, for $n\geq3$. In particular, if $n=2$, we get
 \begin{equation*}
     \Gamma_2=\mathrm{Span}_{\ZZ_{\geq0}} \left \{(0,-1),(1,1),(1,2)\right \}=\mathrm{Span}_{\ZZ_{\geq0}} \left \{(0,-1),(1,2)\right \}.
 \end{equation*} Thus, $X_{\Gamma_2}=\CC^2$.
 \item[Description of $\Gamma_3$, where $m_3=(n+1,n+1)$] By definition,
 \begin{align*}
     \Gamma_3&=\Gamma+\mathrm{Span}_{\ZZ_{\geq0}} \left \{(-n+1,-n-1),(-n+1,-n)\right \}\\
     &=\mathrm{Span}_{\ZZ_{\geq0}} \left \{(1,0),(1,1),(n,n+1),(-n+1,-n-1),(-n+1,-n)\right \}.
 \end{align*}
Since the rays of $\Gamma_3$ are $(-n+1,-n),(n,n+1)$ and  $$\begin{vmatrix}
 -n+1 & n \\
 -n & n+1
\end{vmatrix}=1.$$  Thus, $X_{\Gamma_3}=\CC^2$.

\end{description}
In the special case when $n=1$, we have that $I=\{(2,0),(2,1),(2,2)\}$. Notice that $m_2=(2,1)$ is not a vertex of $\mathcal{N}(I)$. Thus, $\mathrm{Bl}_{Der\,}X_\Gamma$ only has two affine charts, both are isomorphic to $\CC^2$.
\end{proof}

\begin{corollary}
\label{Cor:An} Iterating the blow-ups at the sheaf of derivations produces a resolution of $(A_n,0)$. Moreover, it follows that $n/2$ blow-ups are required when $n$ is even, and $(n+1)/2$ when $n$ is odd. Furthermore, upon completion of this process, one obtains the minimal resolution. 
\end{corollary}

\begin{remark}Let us note that in the proof of Theorem \ref{an}, each blow-up gives rise to two divisors, so after iterating the process, we obtain exactly n divisors, which results in the minimal resolution. This will be described below in a more general and precise manner (see Theorem \ref{th:Main.rational}). 
\end{remark}

\subsection{Rational surface singularities}
In this section, we denote by $(X,x)$ the germ of a rational normal surface singularity. Let $\pi\colon \Rs_{min} \to X$ be the minimal resolution of $X$. Since $(X,x)$ is a rational singularity, the minimal resolution coincides with the minimal good resolution.

\begin{theorem}
\label{th:Main.rational}
    Any rational surface singularity can be resolved after a finite number of blow-ups at the module of derivations. Moreover, such resolution is the minimal resolution of singularities.
\end{theorem}
\begin{proof}
Since $(X,x)$ is a rational singularity and the Zariski-Lipman conjecture is true for rational surface singularities (see \cite{Graf}), the module $\mathrm{Der}_{\CC}(\Ss{X,x})$ is not free. Denote by
\begin{equation*}
    f_1\colon \mathrm{Bl}_{\mathrm{Der}}(X) \to X,
\end{equation*}
the blow-up of $\mathrm{Der}_{\CC}(\Ss{X,x})$ at $X$. Moreover, denote by 
\begin{equation*}
     \rho_1 \colon \widetilde{\mathrm{Bl}_{\mathrm{Der}}(X)}_{\text{min}} \to \mathrm{Bl}_{\mathrm{Der}}(X),
\end{equation*}
the minimal resolution of $\mathrm{Bl}_{\mathrm{Der}}(X)$.

Let $\pi\colon \Rs \to X$ be the minimal resolution of $X$.  We have the following commutative diagram:
\begin{equation}
\label{eq.Diagram1}
\begin{tikzpicture}
  \matrix (m)[matrix of math nodes,
    nodes in empty cells,text height=2ex, text depth=0.25ex,
    column sep=3.5em,row sep=3em] {
    \widetilde{\mathrm{Bl}_{\mathrm{Der}}(X)}_{\text{min}} & \Rs\\
    \mathrm{Bl}_{\mathrm{Der}}(X) & X\\
};
\draw[-stealth] (m-2-1) edge node[below]{$f_1$} (m-2-2);
\draw[-stealth] (m-1-1) edge node[left]{$\rho_1$} (m-2-1);

\draw[-stealth] (m-1-2) edge node[right]{$\pi$} (m-2-2);
\end{tikzpicture}
\end{equation}
Since $X$ has rational singularities and $\mathrm{Der}_{\CC}(\Ss{X,x})$ is a reflexive $\Ss{X}$-module (see Remark \ref{rmk:Der.refl}), by \cite[Lemma and Definition 1.1]{Es} the minimal resolution $\Rs$ is the minimal adapted resolution of $\mathrm{Der}_{\CC}(\Ss{X,x})$. By Proposition~\ref{Prop:2daconstruccion} we get $ \widetilde{\mathrm{Bl}_{\mathrm{Der}}(X)}_{\text{min}} \cong \Rs$. Furthermore, by Theorem~\ref{Theo:IntroBlowUpMMinAdapt} we know that $\mathrm{Bl}_{\mathrm{Der}}(X)$ is obtained by contracting some divisors of $\Rs$.

If $\mathrm{Bl}_{\mathrm{Der}}(X)$  is smooth, we are done. Suppose that $\mathrm{Bl}_{\mathrm{Der}}(X)$ has singularities. By~\cite[Proposition~4.2]{Gus}, $\mathrm{Bl}_{\mathrm{Der}}(X)$ has only normal singularities. Now we prove that $\mathrm{Bl}_{\mathrm{Der}}(X)$ has only rational singularities. This follows from the 
Leray spectral sequence applied to the composition $\pi = f_1 \circ \rho_1$:
\begin{equation}
    E_2^{p,q} = R^p {f_1}_*\!\left(R^q {\rho_1} _*\mathcal{O}_{\widetilde{\mathrm{Bl}_{\mathrm{Der}}(X)}}\right) 
\Rightarrow R^{p+q}\pi_*\mathcal{O}_{\widetilde{X}}.
\end{equation}
Since $X$ has rational singularities, we get $R^1\pi_*\mathcal{O}_{\widetilde{X}} = 0$. 
The five-term exact sequence of the Leray spectral sequence is
\begin{equation}\label{eq:leray}
0 \to R^1 {f_1}_*\mathcal{O}_{\mathrm{Bl}_{\mathrm{Der}}(X)} 
\to R^1\pi_*\mathcal{O}_{\widetilde{X}} 
\to {f_1}_*\!\left(R^1 {\rho_1}_*\mathcal{O}_{\widetilde{X}}\right) 
\to R^2{f_1}_*\mathcal{O}_{\mathrm{Bl}_{\mathrm{Der}}(X)} 
\to R^2\pi_*\mathcal{O}_{\widetilde{X}}.
\end{equation}
We have the following:
\begin{itemize}
\item $R^i\pi_*\mathcal{O}_{\widetilde{X}} = 0$ for $i > 0$, since $X$ has rational singularities.
\item $R^2{f_1}_*\mathcal{O}_{\mathrm{Bl}_{\mathrm{Der}}(X)} = 0$, since $f_1$ is a proper 
birational morphism between surfaces and its fibers have dimension at most one.
\end{itemize}
Substituting into \eqref{eq:leray} yields ${f_1}_*\!\left(R^1 {\rho_1}_*\mathcal{O}_{\widetilde{X}}\right) = 0$. Since $R^1{\rho_1}_*\mathcal{O}_{\widetilde{X}}$ is a skyscraper sheaf supported at the 
isolated singular points of $\mathrm{Bl}_{\mathrm{Der}}(X)$, and the stalk of its 
direct image at any point equals its own stalk, we conclude 
$R^1 {\rho_1}_*\mathcal{O}_{\widetilde{X}} = 0$, i.e., 
$\mathrm{Bl}_{\mathrm{Der}}(X)$ has rational singularities. Denote by $\mathcal{D}_2:=\Ders_\mathbb{C}(\Ss{\mathrm{Bl}_{\mathrm{Der}}(X)})$, the sheaf of derivations of $\mathrm{Bl}_{\mathrm{Der}}(X)$.
By rationality, we get that $\mathcal{D}_2$ is not locally free. Denote by $\mathrm{Bl}_{\mathcal{D}_2}(\mathrm{Bl}_{\mathrm{Der}}(X))$ the blow-up of $\mathrm{Bl}_{\mathrm{Der}}(X)$ at $\mathcal{D}_2$. We get the following diagram
\begin{equation*}
\begin{tikzpicture}
  \matrix (m)[matrix of math nodes,
    nodes in empty cells,text height=2ex, text depth=0.25ex,
    column sep=3.5em,row sep=3em] {\widetilde{\mathrm{Bl}_{\mathcal{D}_2}(\mathrm{Bl}_{\mathrm{Der}}(X))}_{\text{min}} &
    \widetilde{\mathrm{Bl}_{\mathrm{Der}}(X)}_{\text{min}} & \Rs\\
    \mathrm{Bl}_{\mathcal{D}_2}(\mathrm{Bl}_{\mathrm{Der}}(X)) &\mathrm{Bl}_{\mathrm{Der}}(X) & X\\
};
\draw[-stealth] (m-1-2) edge node[auto]{$\cong$} (m-1-3);
\draw[-stealth] (m-2-2) edge node[below]{$f_1$} (m-2-3);
\draw[-stealth] (m-2-1) edge node[below]{$f_2$} (m-2-2);

\draw[-stealth] (m-1-1) edge node[left]{$\rho_2$} (m-2-1);
\draw[-stealth] (m-1-2) edge node[left]{$\rho_1$} (m-2-2);
\draw[-stealth] (m-1-3) edge node[right]{$\pi$} (m-2-3);
\end{tikzpicture}
\end{equation*}
where $\rho_2 \colon \widetilde{\mathrm{Bl}_{\mathcal{D}_2}(\mathrm{Bl}_{\mathrm{Der}}(X))}_{\text{min}} \to \mathrm{Bl}_{\mathcal{D}_2}(\mathrm{Bl}_{\mathrm{Der}}(X))$ is the minimal resolution. Analogously, as in the first blow-up, we get $\widetilde{\mathrm{Bl}_{\mathcal{D}_2}(\mathrm{Bl}_{\mathrm{Der}}(X))}_{\text{min}} \cong \Rs$. Again, the variety $\mathrm{Bl}_{\mathcal{D}_2}(\mathrm{Bl}_{\mathrm{Der}}(X))$ is obtained by contracting some divisors of $\Rs$ (different from the divisors of the first blow-up).

Since $\Rs$ has a finite number of divisors and every blow-up eliminates some divisors, this process must terminate after finitely many steps. This proves the theorem.
\end{proof}

\begin{remark}
Let $(X,x)$ be the germ of a normal surface singularity. Suppose that $(X,x)$ has a rational singularity. Let $\pi\colon \Rs_{min} \to X$ be the minimal resolution of $X$. Suppose that the exceptional divisor $E$ has $n$ irreducible components. By Theorem~\ref{th:Main.rational}, after a finite number of blow-ups at the module of derivations, we get a smooth variety. Suppose that we need $m$ blow-ups. Then $m \leq n$. Indeed,  by Theorem~\ref{Theo:IntroBlowUpMMinAdapt} we know that each blow-up is obtained by contracting some divisors of $\Rs$. If the module of derivations is special and indecomposable, then by~\cite{Wu} we have to contract all divisors except one. If not, then we have to contract all but at least one divisor. 

Therefore, in the case of rational surface singularities, at most $n$ blow-ups at the module of derivations are needed to resolve the singularity. This bound is
generally not sharp, and the deviation from sharpness has two independent sources. First,
$\mathrm{Der}(\mathcal{O}_X)$ may decompose as a direct sum,
in which case the associated full sheaf detects, at each step, the components singled out by
all of its summands at once. Second, by \cite{Wu} the correspondence between irreducible components of $E$ and indecomposable reflexive modules holds only for the special indecomposable modules. Thus, even if $\mathrm{Der}(\mathcal{O}_X)$ is indecomposable, if it is not special, its first Chern class may intersect several exceptional components simultaneously. Both phenomena reduce the number of blow-ups actually required. For the singularity $A_n$, the module of derivations splits as a direct sum of two special reflexive modules of rank one; by \cite{Wu} each of these detects exactly one exceptional component, so each blow-up detects two components and the process terminates after $\lceil n/2 \rceil$ blow-ups (see Corollary~\ref{Cor:An}) rather than the $n$ given by the general bound.
\end{remark}

\section{Blow-up of log-canonical surface singularities}\label{Sec:LogC}
In this section, we study the blow-up of the module of derivations for log-canonical surface singularities. Let $(X,x)$ be a log-canonical surface singularity. Let $\pi \colon \Rs \to X$ be the minimal resolution. It is known that any log-canonical singularity is either rational, cusp or simple elliptic (see e.g. \cite[Example 6.3.33(d)]{NeBook}). By Theorem~\ref{th:Main.rational}, we only need to study the case of cusps and simple elliptic singularities. Recall that the Zariski-Lipman conjecture is true for log-canonical singularities (see \cite{Graf}).

\subsection{(Generalized) Simple elliptic singularities}
Let $(X,x)$ be a simple elliptic singularity of type $El(b)$, that is, a normal
surface singularity for which the exceptional curve $E$ in the minimal resolution $\pi \colon \Rs_{min} \to X$ is a smooth elliptic curve with self-intersection number $E^2 = -b$ with $b > 1$ (see e.g. \cite[Example 7.2.18]{NeBook}).

Since $(X,x)$ is a simple elliptic singularity, it is quasihomogeneous. Therefore, by \cite[Theorem 2.8]{KBehn} the Auslander module is isomorphic to $\left(\Omega^1_X\right)^{\smvee \smvee}$. Now, by \cite[pp. 216]{KBehn} the sheaf $\pi^* \left(\Omega^1_X\right)^{\smvee \smvee} / \mathrm{tor}$ is locally free. Therefore, $\Rs_{min}$ is the minimal adapted resolution of $\left(\Omega^1_X\right)^{\smvee \smvee}$. Since the exceptional divisor is irreducible, by Theorem~\ref{Theo:IntroBlowUpMMinAdapt} the normalization of the blow-up of $X$ at $\left(\Omega^1_X\right)^{\smvee \smvee}$ is the minimal resolution of $X$. We have proved the following theorem.

\begin{theorem}
    Let $(X,x)$ be a simple elliptic singularity of type $El(b)$. Then, the normalization of the blow-up of $X$ at $\left(\Omega^1_X\right)^{\smvee \smvee}$ is the minimal resolution of $X$.
\end{theorem}
\begin{remark}
    It is expected that a similar statement holds for the blow-up at the module of derivations, although this case has not yet been established.
\end{remark}
\subsection{Cusp singularities}
Let $(X,x)$ be a cusp singularity, i.e., the minimal resolution graph is a cyclic graph with $g(E_j) = \delta(E_j) = 0$ with $j=1,2,\dots,r$ with $r \geq 3$, where $\delta$ denotes the delta invariant (see e.g. \cite[Example 7.2.19]{NeBook} or \cite[4.21. Cusp Singularities]{Ne}). By \cite[Proof of 2.3]{KBehn}, the module of Zariski differentials, i.e., $\left(\Omega^1_X\right)^{\smvee \smvee}$ is the direct sum of two reflexive modules $L_1$ and $L_2$ of
rank one:
\begin{equation*}
    \left(\Omega^1_X\right)^{\smvee \smvee} = L_1 \oplus L_2,\quad \quad
    \mathrm{Der}_{\CC}(\Ss{X,x}) =  L_1^{\smvee} \oplus L_2^{\smvee}.
\end{equation*}
Under these assumptions, we have our main theorem.
\begin{theorem}
\label{th.main.cusp}
    The normalization of the blow-up of $\mathrm{Der}_{\CC}(\Ss{X,x})$ is dominated by the minimal resolution of $X$. The same statement is true for the module of Zariski differentials.
\end{theorem}
\begin{proof}
Let $\pi \colon \tilde{X} \to X$ be the minimal resolution. First, we will prove that $\tilde{X}$ is the minimal adapted resolution of $\mathrm{Der}_{\CC}(\Ss{X,x})$ and $\left(\Omega^1_X\right)^{\smvee \smvee}$. 

We will only consider the case of $\mathrm{Der}_{\CC}(\Ss{X,x})$, an analogous proof works for $\left(\Omega^1_X\right)^{\smvee \smvee}$. We need to prove that $\left( \pi^* \mathrm{Der}_{\CC}(\Ss{X,x}) \right)^{\smvee \smvee}$ is generated by global sections. Since pullback and direct sum commute, it is enough to prove that $\mathcal{L}_1:=\left( \pi^* L_1 \right)^{\smvee \smvee}$ is generated by global sections (the same proof works for $L_2$). By Remark \ref{remark:Son.iguales} we need to prove that $\Sf{L}_1$ does not have any base point. The proof proceeds by contradiction: Suppose that there exist $p_1,\dots,p_n$ base points for $\mathcal{L}_1$.  Recall that $Z_K=Z_{\mathrm{min}}=E$ (see~\cite[Example 7.2.19]{NeBook}).

The proof follows by cases and subcases:
\begin{description}
    \item[Case 1] Suppose that there are at least two base points $p_1$ and $p_2$. Denote by $\rho \colon \Rs' \to \Rs$ the blow-up at $p_1,p_2$. Therefore, we have the following commutative diagram
\begin{equation*}
\begin{tikzpicture}
  \matrix (m)[matrix of math nodes,
    nodes in empty cells,text height=1.5ex, text depth=0.25ex,
    column sep=2.5em,row sep=2em]{
  \Rs' & \Rs \\
  & X \\};
\draw[-stealth] (m-1-1) edge node [above] {$\rho$} (m-1-2);
\draw[-stealth] (m-1-1) edge node [right] {$\pi'$} (m-2-2);
\draw[-stealth] (m-1-2) edge node [right] {$\pi$} (m-2-2);
\end{tikzpicture}
\end{equation*}
Let $\Sf{L}_1'=(\pi'^*L_1)^{\smvee\smvee}$ the full $\Ss{\tilde{X}'}$-module associated to $L_1$. By \cite[Lemma 3.32]{BoRo}, the new divisors $E_{p_1}$ and $E_{p_2}$ of $\Rs'$ do not contain either zeros or poles of the Gorenstein form. Thus, $\Rs'$ is the minimal adapted resolution of $L_1$ (see~\cite[Proposition 3.33, 4)]{BoRo} and \cite[Lemma 3.32]{BoRo}, or the proof of \cite[Proposition 4.1]{BoRo}). By \cite[Proposition 4.15]{BoRo}, we get a contradiction since $L_1$ is indecomposable (since it is of rank one) and a generic global section of $\Sf{L}_1'$ would hit the divisors $E_{p_1}$ and $E_{p_2}$.
\item[Case 2] There is only one base point $p$. Since 
\begin{equation*}
    \Ext^1_{\Rs}(\Sf{L}_1,\Sf{L}_1)\cong \Ext^1_{\Rs}(\Ss{\Rs},\Sf{L}_1\otimes \Sf{L}_1^{\smvee}) \cong \Ext^1_{\Rs}(\Ss{\Rs},\Ss{\Rs})\cong H^1(\Rs,\Ss{\Rs})\cong \mathbb{C},
\end{equation*}
there exists a non-trivial element of $\Ext^1_{\Rs}(\Sf{L}_1,\Sf{L}_1)$. Denote it by
\begin{equation}
\label{eq:Def.L1}
    0 \to \Sf{L}_1 \to \mathcal{H} \to \Sf{L}_1 \to 0.
\end{equation}
We claim that $\Sf{H}$ is a full sheaf. Let us prove this assertion using Proposition~\ref{fullcondiciones}. By \eqref{eq:Def.L1}, we get $\Sf{H}\cong \Sf{H}^{\smvee \smvee}$. Therefore, $\Sf{H}$ is reflexive. Since $\Rs$ is a smooth surface, any reflexive sheaf is locally free (see \cite[Corollary 1.4]{Har3}). Therefore,  $\Sf{H}$ is locally free. By the proof of \cite[Theorem 4.1]{NeRo}, we get 
\begin{equation}
\label{eq:H1.L1.0}
    H^1(\Rs, \Sf{L}_1)=H_E^1(\Rs, \Sf{L}_1)=0.
\end{equation}
Therefore, $H^1(\Rs, \Sf{H})=H_E^1(\Rs, \Sf{H})=0$. Thus, we only need to prove that $\Sf{H}$ is generically generated by global sections. By \eqref{eq:Def.L1} and using \eqref{eq:H1.L1.0}, we get an exact sequence
\begin{equation}
\label{eq:Def.L1.H}
    0 \to H^0(\Rs,\Sf{L}_1) \xrightarrow{\phi} H^0(\Rs,\mathcal{H}) \xrightarrow{\psi} H^0(\Rs,\Sf{L}_1) \to 0.
\end{equation}
Let $l_1, \dots, l_r$ be a collection of sections of $ H^0(\Rs,\Sf{L}_1)$ that generate $\Sf{L}_1$ everywhere except at the point $p$. Denote by $l'_1, \dots, l'_r$ sections of $H^0(\Rs,\mathcal{H})$ such that $\psi(l'_j)=l_j$ for any $j=1,\dots,r$. Now, consider the following commutative diagram
\begin{equation*}
\begin{tikzpicture}
  \matrix (m)[matrix of math nodes,
    nodes in empty cells,text height=2ex, text depth=0.25ex,
    column sep=3.5em,row sep=3em] {0&\Ss{\Rs}^r & \Ss{\Rs}^{2r}& \Ss{\Rs}^r&0\\
    0&\Sf{L}_1 & \Sf{H}& \Sf{L}_1&0\\
};
\draw[-stealth] (m-1-1) edge node[auto]{} (m-1-2);
\draw[-stealth] (m-1-2) edge node[auto]{} (m-1-3);
\draw[-stealth] (m-1-3) edge node[auto]{} (m-1-4);
\draw[-stealth] (m-1-4) edge node[auto]{} (m-1-5);

\draw[-stealth] (m-2-1) edge node[below]{} (m-2-2);
\draw[-stealth] (m-2-2) edge node[above ]{$\phi$} (m-2-3);
\draw[-stealth] (m-2-3) edge node[above ]{$\psi$} (m-2-4);
\draw[-stealth] (m-2-4) edge node[below]{} (m-2-5);

\draw[-stealth] (m-1-2) edge node[left]{$\mathbf{l}$} (m-2-2);
\draw[-stealth] (m-1-3) edge node[right]{$\mathbf{\phi(l),l'}$} (m-2-3);
\draw[-stealth] (m-1-4) edge node[left]{$\mathbf{l}$} (m-2-4);

\end{tikzpicture}
\end{equation*}
where the map $\mathbf{l}$ and $\mathbf{\phi(l),l'}$ are induced by the sections $l_1, \dots, l_r$ and $\phi(l_1), \dots, \phi(l_r),l'_1, \dots, l'_r$, respectively. Let $q$ be a point different from $p$. Since the map $\mathbf{l}$ is onto at the point $q$, by the snake lemma we get that the map $\mathbf{\phi(l),l'}$ is onto at the point $q$. Hence, the sheaf $\Sf{H}$ is generated by global sections everywhere except (possibly) at the point $p$. Therefore, $\Sf{H}$ is a full sheaf of rank $2$. Let $H:= \pi_* \Sf{H}$. 

Now we claim that the full sheaf $\Sf{H}$ is not generated by global sections at the point $p$. We prove the claim by contradiction. Suppose that $\Sf{H}$ is generated by global sections. Let $h_1\dots,h_k$ be sections of $H^0(\Rs,\mathcal{H})$ that generate $\Sf{H}$. Hence, we get the following diagram
\begin{equation*}
\begin{tikzpicture}
  \matrix (m)[matrix of math nodes,
    nodes in empty cells,text height=2ex, text depth=0.25ex,
    column sep=5.5em,row sep=3em] {\Ss{\Rs}^k& \Sf{H}&0\\
    \Ss{\Rs}^k& \Sf{L}_1&\\
        & 0&\\
};
\draw[-stealth] (m-1-1) edge node[auto]{\tiny $h_1,\dots,h_k$} (m-1-2);
\draw[-stealth] (m-1-2) edge node[auto]{} (m-1-3);
\draw[-stealth] (m-2-1) edge node[auto]{\tiny $\psi(h_1),\dots,\psi(h_k)$} (m-2-2);

\draw[-stealth] (m-1-1) edge node[left]{$=$} (m-2-1);
\draw[-stealth] (m-1-2) edge node[right]{$\psi$} (m-2-2);
\draw[-stealth] (m-2-2) edge node[left]{} (m-3-2);
\end{tikzpicture}
\end{equation*}
Therefore, $\Sf{L}_1$ is generated by the global sections $\psi(h_1),\dots,\psi(h_k)$. This is a contradiction since $p$ is a base point of $\Sf{L}_1$.

Hence, $\Sf{H}$ is a full sheaf of rank $2$ that is generated by global sections everywhere except at the point $p$.  Denote by $\rho \colon \Rs' \to X$ the blow-up at $p$. Therefore, we have the following commutative diagram
\begin{equation*}
\begin{tikzpicture}
  \matrix (m)[matrix of math nodes,
    nodes in empty cells,text height=1.5ex, text depth=0.25ex,
    column sep=2.5em,row sep=2em]{
  \Rs' & \Rs \\
  & X \\};
\draw[-stealth] (m-1-1) edge node [above] {$\rho$} (m-1-2);
\draw[-stealth] (m-1-1) edge node [right] {$\pi'$} (m-2-2);
\draw[-stealth] (m-1-2) edge node [right] {$\pi$} (m-2-2);
\end{tikzpicture}
\end{equation*}
Denote by $E_p$ the exceptional divisor of the blow-up at $p$. In the resolution $\Rs'$ the Gorenstein form does not have zeros or poles on $E_p$, thus $\Rs'$ is the minimal adapted resolution of $L_1$ and also of $H$. Hence, $L_1$ and $H$ are special $\Ss{X}$-modules with the same minimal adapted resolution. If $\Sf{H}$ is indecomposable, we get a contradiction with \cite[Corollary 6.12]{BoRo}. The other possibility is that $\Sf{H}$ decomposes into a direct sum of two full sheaves. Since $\Rs'$ is the minimal adapted resolution of $L_1$ and also of $H$  and by \cite[Theorem 6.10]{BoRo}, we get that $\Sf{H}\cong \Sf{L}_1\oplus\Sf{L}_1$. Hence, the exact sequence \eqref{eq:Def.L1} is actually
\begin{equation}
\label{eq:Def.L2}
    0 \to \Sf{L}_1 \stackrel{\phi}{\longrightarrow} \Sf{L}_1\oplus\Sf{L}_1 \stackrel{\psi}{\longrightarrow} \Sf{L}_1 \to 0.
\end{equation}
Applying the functor $- \otimes \Sf{L}_1^{\smvee}$ we get the exact sequence \eqref{eq:Def.L2}
\begin{equation}
\label{eq:Def.L3}
    0 \to \Ss{\Rs} \stackrel{\phi}{\longrightarrow} \Ss{\Rs}\oplus\Ss{\Rs} \stackrel{\psi}{\longrightarrow} \Ss{\Rs} \to 0.
\end{equation}
Note that we abuse the notation. The morphisms of \eqref{eq:Def.L3} are actually $\phi \otimes 1$ and $\psi \otimes 1$. Applying the functor $\pi_* -$ to the exact sequence \eqref{eq:Def.L3} we get the exact sequence
\begin{equation}
\label{eq:Def.L4}
    0 \to \Ss{X} \stackrel{\phi}{\longrightarrow} \Ss{X}\oplus\Ss{X} \stackrel{\psi}{\longrightarrow} \Ss{X} \to R^1\pi_* \Ss{X} \stackrel{\phi}{\longrightarrow} R^1\pi_* \left(\Ss{X}\oplus\Ss{X}\right) \stackrel{\psi}{\longrightarrow} R^1\pi_* \Ss{X} \to 0.
\end{equation}
Recall that,
\begin{equation*}
     R^1\pi_* \Ss{X} \cong H^1(\Rs, \Ss{X}) \cong \mathbb{C} \quad \text{and} \quad  R^1\pi_* \left(\Ss{X}\oplus\Ss{X}\right) \cong \mathbb{C}^2.
\end{equation*}
Thus, the long exact sequence \eqref{eq:Def.L4} becomes
\begin{equation}
\label{eq:Def.L5}
    0 \to \Ss{X} \stackrel{\phi}{\longrightarrow} \Ss{X}\oplus\Ss{X} \stackrel{\psi}{\longrightarrow} \Ss{X} \to 0.
\end{equation}
Since the exact sequence \eqref{eq:Def.L5} is an element of $\Ext^1(\Ss{X},\Ss{X})$ and this group is trivial, there exists an  isomorphism $\theta \colon \Ss{X}\oplus\Ss{X} \to \Ss{X}\oplus\Ss{X}$ such that
\begin{equation*}
\begin{tikzpicture}
  \matrix (m)[matrix of math nodes,
    nodes in empty cells,text height=2ex, text depth=0.25ex,
    column sep=3.5em,row sep=3em] {
    0&\Ss{X} & \Ss{X}\oplus\Ss{X}& \Ss{X}&0\\
    0&\Ss{X} & \Ss{X}\oplus\Ss{X}& \Ss{X}&0\\
};
\draw[-stealth] (m-1-1) edge node[auto]{} (m-1-2);
\draw[-stealth] (m-1-2) edge node[auto]{} (m-1-3);
\draw[-stealth] (m-1-3) edge node[auto]{} (m-1-4);
\draw[-stealth] (m-1-4) edge node[auto]{} (m-1-5);

\draw[-stealth] (m-2-1) edge node[below]{} (m-2-2);
\draw[-stealth] (m-2-2) edge node[above ]{$\phi$} (m-2-3);
\draw[-stealth] (m-2-3) edge node[above ]{$\psi$} (m-2-4);
\draw[-stealth] (m-2-4) edge node[below]{} (m-2-5);

\draw[-stealth] (m-1-2) edge node[left]{$=$} (m-2-2);
\draw[-stealth] (m-1-3) edge node[right]{$\theta$} (m-2-3);
\draw[-stealth] (m-1-4) edge node[left]{$=$} (m-2-4);

\end{tikzpicture}
\end{equation*}
By the isomorphism $\theta$, we get that the exact sequence \eqref{eq:Def.L5} is the trivial extension. Therefore, \eqref{eq:Def.L2} is also the trivial extension. This is a contradiction since we started with a non-trivial extension \eqref{eq:Def.L1}.
\end{description}
Therefore, the only possibility is that $\Sf{L}_1$ does not have any base points. By Remark~\ref{remark:Son.iguales}, $\Rs$ is the minimal adapted resolution of $L_1$. Thus, by Theorem~\ref{Theo:IntroBlowUpMMinAdapt} the normalization of the blow-up of $X$ at $L_1$ is dominated by $\Rs$. This proves the theorem.
\end{proof}
\begin{corollary}

Let $X$ be a cusp singularity. Then the blow-up of $X$ at $\mathrm{Der}_{\CC}(\Ss{X})$, followed by normalization, is either smooth or has only quotient (rational) singularities. In the latter case, a further finite number of blow-ups at the module of derivations resolves $X$.
\end{corollary}
\begin{proof}
Denote by $\mathrm{NBl}_{\mathrm{Der}}(X)$ the normalization of the blow-up of $X$ at $\mathrm{Der}_{\CC}(\Ss{X,x})$. If $\mathrm{NBl}_{\mathrm{Der}}(X)$ is smooth, we are done. Suppose that $\mathrm{NBl}_{\mathrm{Der}}(X)$  is not smooth. By Proposition~\ref{Prop:2daconstruccion} and Theorem~\ref{th.main.cusp}, we get the following diagram
    \begin{equation*}
\begin{tikzpicture}
  \matrix (m)[matrix of math nodes,
    nodes in empty cells,text height=2ex, text depth=0.25ex,
    column sep=3.5em,row sep=3em] {
    \widetilde{\mathrm{NBl}_{\mathrm{Der}}(X)}_{\text{min}} & \Rs\\
    \mathrm{NBl}_{\mathrm{Der}}(X) & X\\
};
\draw[-stealth] (m-1-1) edge node[below]{$\cong$} (m-1-2);
\draw[-stealth] (m-2-1) edge node[below]{$f_1$} (m-2-2);
\draw[-stealth] (m-1-1) edge node[left]{$\rho_1$} (m-2-1);

\draw[-stealth] (m-1-2) edge node[right]{$\pi$} (m-2-2);
\end{tikzpicture}
\end{equation*}
where $\widetilde{\mathrm{NBl}_{\mathrm{Der}}(X)}_{\text{min}}$ is the minimal resolution of $\mathrm{NBl}_{\mathrm{Der}}(X)$. Let $y$ be a singular point of $\mathrm{NBl}_{\mathrm{Der}}(X)$. Since $\widetilde{\mathrm{NBl}_{\mathrm{Der}}(X)}_{\text{min}} \cong \Rs$ and all the components of $E$ form a cycle of rational smooth curves, then $\rho^{-1}(y)$ is a chain of rational smooth curves. By~\cite[Theorem 7.4.17]{Ishi}, $y$ is a quotient singularity. Hence $\mathrm{NBl}_{\mathrm{Der}}(X)$ has only quotient singularities. Thus, the result follows from Theorem~\ref{th:Main.rational}.
\end{proof}
\section*{Acknowledgments} We thank Professor Daniel Duarte and Professor Mark Spivakovsky for useful conversations and for encouraging this work. We thank the referees for their suggestions. The first named author was supported by Universidad Nacional Autónoma de México Postdoctoral Program (POSDOC). The first and second-named authors were supported by SECIHTI project CF-2023-G-33. The third-named author was supported by DGAPA PAPIIT IN112424 Complex Kleinian groups,  DGAPA PAPIIT IN117523 Singularidades de superficies complejas: modificaciones, resoluciones y curvas polares and DGAPA PAPIIT IA102626 Singularidades complejas y reales: Invariantes analíticos y topológicos. The authors were supported by Simons-UNAM Geometry Program at Cuernavaca
Simons Foundation International SFI-MPS-T-Institutes-00011977 JS.

\end{document}